\documentclass[runningheads]{llncs}
\usepackage{graphicx}
\usepackage{amsmath}
\usepackage{amssymb}

\begin{document}
\title{Logarithmic divergences: geometry and interpretation of curvature}
%
%

\author{Ting-Kam Leonard Wong\inst{1} \and 
Jiaowen Yang\inst{2}}
\authorrunning{T.-K.~L.~Wong and J.~Yang}
%
\institute{Department of Statistical Sciences, University of Toronto\\
\email{tkl.wong@utoronto.ca} \and
Department of Mathematics, University of Southern California\\
\email{jiaoweny@usc.edu}}
\maketitle              
\begin{abstract}
We study the logarithmic $L^{(\alpha)}$-divergence which extrapolates the Bregman divergence and corresponds to solutions to novel optimal transport problems. We show that this logarithmic divergence is equivalent to a conformal transformation of the Bregman divergence, and, via an explicit affine immersion, is equivalent to Kurose's geometric divergence. In particular, the $L^{(\alpha)}$-divergence is a canonical divergence of a statistical manifold with constant sectional curvature $-\alpha$. For such a manifold, we give a geometric interpretation of its sectional curvature in terms of how the divergence between a pair of primal and dual geodesics differ from the dually flat case. Further results can be found in our follow-up paper \cite{WY19} which uncovers a novel relation between optimal transport and information geometry.

\keywords{Logarithmic divergence \and Bregman divergence \and Conformal divergence \and Affine immersion \and Constant sectional curvature \and Optimal transport}
\end{abstract}

\section{Introduction}
Let $\Omega \subset \mathbf{R}^n$ be an open convex set, $n \geq 2$. For $\alpha > 0$ fixed, we say that a function $\varphi: \Omega \rightarrow \mathbf{R}$ is $\alpha$-exponentially concave if $e^{\alpha \varphi}$ is concave on $\Omega$. All functions in this paper are assumed to be smooth. Given such a function $\varphi$, we define its $L^{(\alpha)}$-divergence by
\begin{equation} \label{eqn:L.alpha.divergence}
\mathbf{L}^{(\alpha)}_{\varphi}[\xi : \xi'] := \frac{1}{\alpha} \log (1 + \alpha D \varphi(\xi') \cdot (\xi - \xi')) - (\varphi(\xi) - \varphi(\xi')), \quad \xi, \xi' \in \Omega,
\end{equation}
where $D\varphi$ is the Euclidean gradient and $\cdot$ is the dot product. We always assume the Hessian $D^2 e^{\alpha \varphi}$ is strictly negative definite on $\Omega$. Then $\mathbf{L}^{(\alpha)}_{\varphi}$ is a divergence on $\Omega$, regarded as a manifold, in the sense of \cite[Definition 1.1]{A16}. As $\alpha \downarrow 0$, the $L^{(\alpha)}$-divergence (with $\varphi$ fixed) converges to the Bregman divergence defined by
\begin{equation} \label{eqn:Bregman.divergence}
\mathbf{B}_{\phi}[\xi : \xi'] := (\phi(\xi) - \phi(\xi')) - D\phi(\xi') \cdot (\xi - \xi'),
\end{equation}
where $\phi = -\varphi$ is convex with $D^2 \phi > 0$. Thus the family $\{ {\bf L}_{\varphi}^{(\alpha')} \}_{0 < \alpha' \leq \alpha}$ of logarithmic divergences extrapolate the Bregman divergence $\mathbf{B}_{\phi}$. 

Originally motivated by applications in stochastic portfolio theory \cite{F02}, the $L^{(1)}$-divergence (and its extension to the $L^{(\alpha)}$-divergence) was introduced by Pal and the first author in \cite{PW16} \cite{W19} and was studied further in \cite{W15} \cite{PW18} \cite{W18}. There are two main results proved in these papers. First, the $L^{(\alpha)}$-divergence corresponds to the solution to an optimal transport problem with a logarithmic cost function; this is formulated using the general framework of $c$-divergence, see \cite{PW18} \cite{W18} \cite{WY19}. Also see  \cite{PW18b} \cite{KZ19} \cite{P19} for recent results about the optimal transport problem which have independent mathematical interest. Second, the induced statistical manifold $(\mathcal{M}, g, \nabla, \nabla^*)$ (see \cite[Section 6.2]{A16} for the definition) is dually projectively flat with constant sectional curvature $-\alpha$. In \cite{W19} we also defined an $L^{(-\alpha)}$-divergence corresponding to constant positive sectional curvature $\alpha$. For expositional simplicity we only consider the $L^{(\alpha)}$-divergence in this paper and \cite{WY19}, but similar results hold for the $L^{(-\alpha)}$-case as well.

In this paper we develop two geometric aspects of the logarithmic divergence. First, we connect the $L^{(\alpha)}$-divergence with classical topics in information geometry, namely conformal transformation and affine differential geometry. In particular, by using an explicit affine immersion, we show that the $L^{(\alpha)}$-divergence is equivalent to the canonical geometric divergence constructed by Kurose \cite{K94}. Second, we provide a geometric interpretation of the sectional curvature for a statistical manifold with constant negative sectional curvature. By analyzing a canonical divergence between a pair of primal and dual geodesics, we show that the sectional curvature can be quantified in terms of the deviation from the generalized Pythagorean relation of a dually flat manifold (see Theorem \ref{thm:curvature} below). This extends the geometric interpretation of sectional curvature in Riemannian geometry. In our follow-up work \cite{WY19} we proved a more general result (see \cite[Theorem 3.13]{WY19}) that holds for any divergence (though it is not intrinsic in the information geometric sense). This was achieved by a novel relation between information geometry and the pseudo-Riemannian framework of Kim and McCann \cite{KM10} concerning the Ma-Trudinger-Wang tensor in optimal transport.

\section{Conformal divergence and its geometry}
We refer the reader to \cite{A16} for general background in information geometry. Conformal transformations of divergence have been studied in the literature; see for example \cite{OAT91} \cite{M99} \cite{AC10} \cite{NNA16} and the references therein. An important application is robust clustering \cite{VB10} \cite{NN15}. 

\begin{definition} \label{def:conformal.divergence}
Let $\phi : \Omega \rightarrow \mathbb{R}$ be convex (with $D^2 \phi > 0$) and let $\kappa: \Omega \rightarrow (0, \infty)$. We define the (left-sided) conformal transformation of the Bregman divergence $\mathbf{B}_{\phi}$ by
\begin{equation} \label{eqn:conformal.divergence}
\mathbf{D}_{\phi,\kappa}[\xi : \xi'] := \kappa(\xi) \mathbf{B}_{\phi}[\xi : \xi'].
\end{equation}
To abbreviate we call $\mathbf{D}_{\phi, \kappa}$ a conformal divergence.
\end{definition}

Note that a right-sided conformal transformation can be converted to a left-sided one by considering the convex conjugate of $\phi$ (see \cite[p.17]{A16}).

Our first result is that the $L^{(\alpha)}$-divergence is, up to a monotone transformation, equal to a conformal transformation of a Bregman divergence. This shows that the geometry induced by the $L^{(\alpha)}$-divergence can be studied using results of Bregman divergence and conformal transformation. 

\begin{theorem} \label{thm:l.c.equi}
Consider an $L^{(\alpha)}$-divergence $\mathbf{L}_{\varphi}^{(\alpha)}$ on $\Omega$ as in \eqref{eqn:L.alpha.divergence}. Let $\phi = - e^{\alpha \varphi}$ which is convex and let $\kappa = -\frac{1}{\alpha \phi} > 0$. Then, with $T(x) = \frac{1}{\alpha} (e^{\alpha x} - 1)$, we have
\begin{equation} \label{eqn:L.alpha.as.conformal}
T (\mathbf{L}_{\varphi}^{(\alpha)}) \equiv \mathbf{D}_{\phi, \kappa}.
\end{equation}
In particular, the conformal divergence $\mathbf{D}_{\phi, h}$ induces the same dualistic structure $(g, \nabla, \nabla^*)$ as that of $\mathbf{D}_{\varphi}^{(\alpha)}$.
\end{theorem}

\begin{proof}
The identity \eqref{eqn:L.alpha.as.conformal}, once conceived, can be verified by a straightforward computation. The second statement is a consequence of the following lemma which can be proved again by a computation. Note that similar reasonings are used in \cite[Lemma 3]{W18} and \cite[Theorem 17]{W18}. \qed
\end{proof}

\begin{lemma}\label{lemma:equi.geo}
Let $\tilde{\mathbf{D}}$ and $\mathbf{D}$ be divergences related by a monotone transformation: $\tilde{\mathbf{D}} = T(\mathbf{D})$, where $T: [0, \infty) \rightarrow [0, \infty)$ is strictly increasing with $T(0) = 0$. Let $(g, \nabla, \nabla^*)$ and $(\tilde{g}, \tilde{\nabla}, \tilde{\nabla}^*)$ be respectively the dualistic structures induced by $\mathbf{D}$ and $\tilde{\mathbf{D}}$. Then, in any local coordinate system, the coefficients of the dualistic structures are related by
\begin{equation} \label{eqn:monotone.transformation.divergence}
\tilde{g}_{ij} = T'(0) g_{ij}, \quad \tilde{\Gamma}_{ijk} = T'(0) \Gamma_{ijk}, \quad \tilde{\Gamma}_{ijk}^* = T'(0) \Gamma_{ijk}^*.
\end{equation}
In particular, we have $\tilde{\Gamma}_{ij}\mathstrut^{k} = \Gamma_{ij}\mathstrut^{k}$ and $\tilde{\Gamma}_{ij}^*\mathstrut^{k} = \Gamma_{ij}\mathstrut^{k}$, and the primal and dual curvature tensors are the same.
\end{lemma}

\begin{remark}
By Lemma \ref{lemma:equi.geo}, we say that two divergences ${\bf D}$ and $\tilde{{\bf D}}$ are equivalent if there exists $T$ (as in Lemma \ref{lemma:equi.geo} with $T'(0) = 1$) such that $\tilde{\mathbf{D}} = T(\mathbf{D})$. Clearly this defines an equivalence relation among divergences on a manifold. Theorem \ref{thm:l.c.equi} thus states that the $L^{(\alpha)}$-divergence is equivalent to a conformal divergence.
\end{remark}

Theorem \ref{thm:l.c.equi} motivates us to study conformal divergences in general. Recall that two torsion-free affine connections $\nabla$ and $\tilde{\nabla}$ are projectively equivalent if there exists a $1$-form $\tau$ such that
\[
\nabla_X Y = \tilde{\nabla}_Y X + \tau(X)Y + \tau(Y)X
\]
for any vector fields $X$ and $Y$. For its geometric interpretation see \cite[p.17]{NS94}. In particular, $\nabla$ and $\tilde{\nabla}$ have the same geodesics up to time reparameterizations. By definition, $\nabla$ is projectively flat if it is projectively equivalent to a flat connection. When considering the $L^{(\alpha)}$-divergence or a conformal divergence, we think of $\mathcal{M}$ (equal to $\Omega$ as a set) as a manifold, and $\xi$ is the primal (global) coordinate system with values in the convex set $\Omega$.

\begin{proposition}  \label{prop:conformal}
Let $(\mathcal{M}, g, \nabla, \nabla^*)$ be the statistical manifold induced by a conformal divergence $\mathbf{D}_{\phi, \kappa}$.
\begin{itemize}
\item[(i)] The primal connection $\nabla$ is projectively flat and the primal geodesics are, up to time reparameterization, straightlines in the $\xi$-coordinate system. (In fact, using the language of \cite[Section 8.4]{AN00}, $\nabla$ is $(-1)$-conformally flat and $\nabla^*$ is $1$-conformally flat.)
\item[(ii)] $\nabla$ has constant sectional curvature $\lambda \in \mathbb{R}$ with respect to $g$ if and only if
\begin{equation} \label{eqn:constant.curvature.condition}
\frac{1}{\kappa(\xi)} \equiv \lambda \phi(\xi) + a + \sum_{i = 1}^n b_i \xi^i
\end{equation}
for some real constants $a$ and $b^i$. In this case, the dual sectional curvature is also constant and is equal to $\lambda$.
\end{itemize}
\end{proposition}

\begin{remark}
Note that if \eqref{eqn:constant.curvature.condition} holds then one may absorb the linear terms in the definition of $\phi$. On the other hand, we observe that if $\lambda < 0$ then $\lambda \phi$ is concave. Since on $\mathbb{R}^d$ there are no non-trivial positive concave functions, from \eqref{eqn:constant.curvature.condition} we see that if the sectional curvature is constant and negative, the domain $\Omega$ must be a proper subset of $\mathbb{R}^d$. 
\end{remark}

\begin{proof}[of Proposition \ref{prop:conformal}]
Consider the dualistic structure $(g, \nabla, \nabla^*)$ induced by the conformal divergence. Consider the Euclidean coordinate $\xi$ on $\Omega$. By a direct computation, the coefficients of $g$ and $\nabla$ are given by
\begin{equation} \label{eqn:conformal.coefficients}
\begin{split}
g_{ij}(\xi) &= \kappa(\xi) \partial_i \partial_j \phi(\xi), \\
\Gamma_{ij}\mathstrut^{k}(\xi) &= \frac{\partial_i \kappa(\xi)}{\kappa(\xi)} \delta_j^k + \frac{\partial_j \kappa(\xi)}{\kappa(\xi)} \delta_i^k. \\
\end{split}
\end{equation}
Since $\kappa > 0$, the $1$-form $\tau = d \log \kappa$ is well-defined. From \eqref{eqn:conformal.coefficients}, we have that $\nabla_X Y = \tilde{\nabla}_Y X + \tau(X)Y + \tau(Y)X$ we $\tilde{\nabla}$ is the Euclidean flat connection on $\Omega$. Thus $\nabla$ is projectively flat and we have (i). A further computation shows that
\begin{equation} \label{eqn:conformal.curvature.coefficients}
R_{ijk}\mathstrut^{\ell} (\xi) = \kappa(\xi) \left( \partial_{jk} \frac{1}{\kappa}(\xi) \delta_i^{\ell} - \partial_{ik} \frac{1}{\kappa}(\xi) \delta_j^{\ell}\right).
\end{equation}
Using \eqref{eqn:conformal.curvature.coefficients}, we see that $\nabla$ has constant sectional curvature $\lambda \in \mathbb{R}$ with respect to $g$ (see \cite[Definition 12]{W18}) if and only if
\[
\kappa(\xi) \partial_{jk} \frac{1}{\kappa}(\xi) = \kappa(\xi) g_{jk}(\xi) = \lambda  \kappa(\xi) \partial_{jk}\phi(\xi),
\]
which is equivalent to \eqref{eqn:constant.curvature.condition} after integration. \qed
\end{proof}

\section{Realization by affine immersion} \label{sec:affine.immersion}
Consider a statistical manifold $(\mathcal{M}, g, \nabla, \nabla^*)$. In \cite[Theorem 18]{W18} we proved that if both $\nabla$ and $\nabla^*$ are dually projectively flat with constant sectional curvature $-\alpha < 0$, then one can define intrinsically a local divergence of $L^{(\alpha)}$-type which induces the given geometric structure. In this result, a key idea is that the primal and dual coordinates are related by an optimal transport map (this leads to the self-dual representation given by  \eqref{eqn:L.alpha.self.dual} below). In fact, by \cite[Theorem 8.3]{AN00}, if a statistical manifold has constant sectional curvature, then we automatically have dual projective flatness. Thus the condition about projective flatness is redundant and we may modify the statement as follows:

\begin{theorem} \cite[Theorem 18]{W18} \label{thm:intrinsic}
The $L^{(\alpha)}$-divergence is a (local) intrinsic divergence for a statistical manifold with constant negative sectional curvature.
\end{theorem}

On the other hand, for a (simply connected) statistical manifold with constant sectional curvature, Kurose \cite{K94} defined globally a canonical, intrinsic divergence using affine differential geometry and proved that it satisfies a generalized Pythagorean theorem. In this section we show that if $(\Omega, g, \nabla, \nabla^*)$ is induced by an $L^{(\alpha)}$-divergence ${\bf L}_{\varphi}^{(\alpha)}$, then the geometric divergence is the conformal divergence ${\bf D}_{\phi, \kappa}$ in  \eqref{eqn:L.alpha.as.conformal}. While these canonical divergences are equivalent, our approach in \cite{PW18} \cite{W18} gives an explicit construction in Kurose's work, covers the Bregman and $L^{(\alpha)}$-divergences under the same framework, and suggests previously unknown connections with optimal transport maps. 

\medskip
To state the main result we recall some concepts of affine differential geometry; for details see \cite{NS94} and \cite{Mat10}. Let ${\mathcal{M}}$ be an $n$-dimensional manifold. An affine immersion of $M$ into $\mathbb{R}^{n+1}$ consists of an immersion $f: M \rightarrow \mathbb{R}^{n+1}$ and a transversal vector field ${\bf n}$ with values in $\mathbb{R}^{n+1}$ on $M \cong f(M)$. The last statement means that
\[
T_{f(p)} \mathbb{R}^{n+1} = f_*(T_p M) \oplus \mathrm{span}({\bf n}(p))
\]
for all $p \in M$. Let $\overline{\nabla}$ be the standard (flat) affine connection on $\mathbb{R}^{n+1}$. Then the covariant derivative decomposes as
\begin{equation} \label{eqn:affine.immersion.definition}
\overline{\nabla}_X f_*Y = f_*(\nabla_X Y) + g(X, Y) {\bf n}.
\end{equation}
We call $\nabla$ and $g$ the induced connection and bilinear form respectively. If the induced connection and bilinear form are equal to the Riemannian metric and primal connection of a dualistic structure $(g, \nabla, \nabla^*)$, we say that the affine immersion realizes the given structure. By \cite[Theorem 5.3]{Mat10}, this is possible when the statistical manifold is simply connected and $1$-conformally flat. This is true in particular when the statistical manifold has constant sectional curvature.

Let $(\mathbb{R}^{n+1})^*$ be the dual space of $\mathbb{R}^n$, and let $\langle . , . \rangle$ be the dual pairing. Given an affine immersion $(f, {\bf n})$, the conormal vector field ${\bf n}^* : M \rightarrow (\mathbb{R}^{n+1})^*$ is defined by the conditions
\begin{equation} \label{eqn:conormal.field}
\langle {\bf n}^*(p), {\bf n}(p) \rangle = 1, \quad \langle {\bf n}^*(p), f_* X \rangle = 0 \quad  \forall X \in T_pM.
\end{equation}

\begin{definition}[Kurose's geometric divergence]
	For an affine immersion $(f, {\bf n})$ with conormal field ${\bf n}^*$, the geometric divergence is defined by
	\begin{equation} \label{eqn:geometric.divergence}
	\rho(p, q) := \langle f(p) - f(q), {\bf n}^*(q) \rangle, \quad p, q \in M.
	\end{equation}
\end{definition}

In \cite{K94} it was shown that if $(g, \nabla)$ is $1$-conformally flat, then the geometric divergence does not depend on the choice of the immersion and recovers the given dualistic structure. (The dual connection $\nabla^*$ is uniquely determined given $g$ and $\nabla$.) Hence, it can be viewed as a {\it canonical divergence} (see the next section for more discussion).

The following result connects the $L^{(\alpha)}$-divergence with the geometric divergence. It shows that the geometric divergence, the $L^{(\alpha)}$-divergence and the conformal divergence are all equivalent. In particular, they are all intrinsically defined (at least locally) for the given dualistic structure.

\begin{theorem} \label{thm:Kurose.as.conformal}
	Consider a convex domain $\Omega \subset \mathbb{R}^n$ equipped with an $L^{(\alpha)}$-divergence ${\bf L}_{\varphi}^{(\alpha)}$ and its induced geometry $(g, \nabla, \nabla^*)$. Let $\phi = -e^{\alpha \varphi}$ and $\kappa = -\frac{1}{\alpha \phi}$ as in Theorem \ref{thm:l.c.equi}. Consider the affine immersion defined by
	\begin{equation} \label{eqn:immersion}
	\begin{split}
	f(\xi) &= \kappa(\xi) (\xi^1,\xi^2,...,\xi^n,1),\\
	{\bf n}(\xi) &= \alpha f(\xi),
	\end{split}
	\end{equation}
	where $\xi$ is the Euclidean coordinate system on $\Omega$. Then this affine immersion realizes $(g, \nabla)$. Moreover, the geometric divergence is given by
	\begin{equation} \label{eqn:Kurose.as.conformal}
	\rho(\xi, \xi') = \mathbf{D}_{\phi, \kappa}(\xi, \xi') = \frac{1}{\alpha} \left( e^{\mathbf{L}_{\varphi}^{(\alpha)}[\xi : \xi']} - 1\right).
	\end{equation}
\end{theorem}
\begin{proof}
	The choice of our immersion \eqref{eqn:immersion} is motivated by the proof of \cite[Proposition 2.7]{NS94}. It is easy to see that $f$ is an immersion and ${\bf n}$ is transversal. Let
	$\tilde{{\bf e}}_j := \frac{\partial}{\partial \xi^j} f$ and $\partial_k \tilde{{\bf e}}_j := \frac{\partial}{\partial \xi^k} \frac{\partial}{\partial \xi^j} f$. Then, it can be verified by a straightforward computation that 
	\begin{equation} \label{eqn:immersion.identity}
	\partial_k \tilde{{\bf e}}_j \equiv \Gamma_{kj}\mathstrut^{m} \tilde{{\bf e}}_m + g_{ij} (\alpha f).
	\end{equation}
	We refer the reader to \cite[Section 5]{W18} for expressions of the coefficients ${\Gamma_{ij}}^k$. Thus the affine immersion $(f, {\bf n})$ realizes the given dualistic structure.
	
	Next we construct the conormal vector field. Using the relations in \eqref{eqn:conormal.field}, we can show that the conormal field is given by
	\begin{equation} \label{eqn:conormal.field.expression}
	{\bf n}^*(p_{\xi}) = (- \partial_1 \phi(\xi), \ldots, -\partial_n \phi(\xi), -\alpha \phi(\xi) + D\phi(\xi) \cdot \xi ).
	\end{equation}
	We obtain \eqref{eqn:Kurose.as.conformal} by plugging \eqref{eqn:conormal.field.expression} into \eqref{eqn:Kurose.as.conformal}.
\end{proof}

\section{Interpretation of sectional curvature}
Consider a statistical manifold $(\mathcal{M}, g, \nabla, \nabla^*)$. Given $q \in \mathcal{M}$ and $\mathbf{v}, \mathbf{w} \in T_q \mathcal{M}$ which are linearly independent, we can define the primal sectional curvature $\mathrm{sec}(\mathbf{v}, \mathbf{w})$ by
\begin{equation} \label{eqn:sectional.curvature}
\mathrm{sec}(\mathbf{v}, \mathbf{w}) := \frac{\langle R(\mathbf{w}, \mathbf{v})\mathbf{v}, \mathbf{w} \rangle}{\|\mathbf{v}\|^2 \|\mathbf{w}\|^2 - \langle \mathbf{v}, \mathbf{w}\rangle^2},
\end{equation}
where $\langle . , . \rangle$ is the Riemannian inner product and $R$ is the primal curvature tensor. Similarly, we can define the dual sectional curvature $\mathrm{sec}^*$. What are the geometric interpretations of these sectional curvatures? Interestingly, to the best of our knowledge, this natural question has not been satisfactorily answered in the literature. 

For motivations, let us consider a Riemannian manifold $(\mathcal{M}, g)$. In this case, it is well-known that the sectional curvature (given by \eqref{eqn:sectional.curvature} using the Levi-Civita connection) can be interpreted in terms of the Riemannian distance, defined by
\begin{equation} \label{eqn:Riemannian.distance}
d(x, y) := \inf_{\gamma: \gamma(0) = x, \gamma(1) = y}  \int_0^1 \|\dot{\gamma}(t)\| dt,
\end{equation}
between a pair of geodesics. For $t_1, t_2 > 0$ small, let $r(t_1) = \exp_q(t_1\mathbf{v})$ and $p(t_2) = \exp_q(t_2 \mathbf{w})$ be geodesics starting at $q$, where $\exp_q$ is the exponential map. Then, we have
\begin{equation} \label{eqn:Riemannian.distance.expansion}
    \begin{split}
    &  d^2(r(t_1), p(t_2)) \\
    &= \| \mathbf{v} \|^2 t_1^2 + \| \mathbf{w} \|^2 t_2^2 - 2 \langle \mathbf{v}, \mathbf{w} \rangle t_1 t_2 - \frac{1}{3} \langle R(\mathbf{w}, \mathbf{v})\mathbf{v}, \mathbf{w} \rangle t_1^2 t_2^2 + \cdots,
    \end{split}
\end{equation}
where the higher order terms are omitted (see \cite{S12}). This implies that
\begin{equation} \label{eqn:Riemannian.sectional.curvature.interpretation}
\begin{split}
& d^2(r(t_1), p(t_2)) - d^2 (r(t_1), q) - d^2(q, p(t_2)) \\
&= -2 \langle \mathbf{v}, \mathbf{w} \rangle t_1 t_2 - \frac{1}{3} (\|\mathbf{v}\|^2 \|\mathbf{w}\|^2 - \langle \mathbf{v}, \mathbf{w} \rangle^2) \mathrm{sec}(\mathbf{v}, \mathbf{w}) t_1^2 t_2^2 + \cdots.
\end{split}
\end{equation}

We look for analogous geometric interpretations for a statistical manifold. Given a statistical manifold $(\mathcal{M}, g, \nabla, \nabla^*)$, in order to formulate a statement in the form of \eqref{eqn:Riemannian.distance.expansion} or \eqref{eqn:Riemannian.sectional.curvature.interpretation}, we need to have an intrinsically defined divergence corresponding to the given geometry. This is the problem about constructing a canonical divergence and was studied by several papers including \cite{HK00} \cite{AA15} \cite{FA18} \cite{FA18b}.

Using the $L^{(\alpha)}$-divergence which is explicit, intrinsically defined and has special properties, in this section we study the geometric interpretation for a statistical manifold with constant sectional curvature $-\alpha \leq 0$. Let $q \in \mathcal{M}$ and ${\bf v}, {\bf w} \in T_q \mathcal{M}$. Motivated by the generalized Pythagorean theorem which holds for the Bregman and $L^{(\alpha)}$-divergences, let
\[
r(t_1) = \exp_q(t_1 {\bf v}) \text{ and } p(t_2) = \exp_q^*(t_2 {\bf w}),
\]
where $\exp_q$ and $\exp_q^*$ are respectively the exponential maps corresponding respectively to the primal and dual connections $\nabla$ and $\nabla^*$. With $\mathbf{D}$ being an intrinsic local $L^{(\alpha)}$-divergence (see Theorem \ref{thm:intrinsic}), consider the expression $H$ defined by
\begin{equation} \label{eqn:H}
H(t_1, t_2) := \mathbf{D}[r(t_1) : p(t_2)] - \mathbf{D}[r(t_1) : q] - \mathbf{D}[q : p(t_2)].
\end{equation}
By the generalized Pythagorean theorem proved in \cite[Theorem 1.2]{PW18} and \cite[Theorem 16]{W18}, if $\langle \mathbf{v}, \mathbf{w} \rangle = 0$ then $H(t_1, t_2) \equiv 0$. This motivates the definition of $H$ and the comparison with \eqref{eqn:Riemannian.sectional.curvature.interpretation}. Note that if $\alpha = 0$ then the manifold is dually flat. In this case, there is a canonical divergence $\mathbf{D}$ of Bregman type. With the Bregman divergence and with $H$ defined by \eqref{eqn:H}, we have the identity $H(t_1, t_2) \equiv - \langle \mathrm{v}, \mathrm{w} \rangle t_1 t_2$.

Now let $\alpha > 0$ and let $\mathbf{D}$ be the canonical (local) $L^{(\alpha)}$-divergence. By \cite[Theorem 18]{W18}, there exists a local coordinate system $\xi$ and an $\alpha$-exponentially concave function $\varphi = \varphi(\xi)$ such that $\mathbf{D}[y : x] = \mathbf{D}^{(\alpha)}_{\varphi} [ \xi_y : \xi_x ]$. Here $\xi_x$ is the primal coordinate of $x \in \mathcal{M}$. Moreover, letting
\[
\eta = \frac{D\varphi(\xi)}{1 - \alpha D\varphi(\xi) \cdot \xi}, \quad \psi(\eta) = \frac{1}{\alpha} \log (1 + \alpha \xi \cdot \eta) - \varphi(\xi),
\]
be respectively the dual coordinate and $\alpha$-conjugate of $\varphi$, we have $\mathbf{D}[y : x] = \mathbf{D}^{(\alpha)}_{\psi} [ \eta_x : \eta_y ]$ and the self-dual representation
\begin{equation} \label{eqn:L.alpha.self.dual}
\mathbf{D}[y : x] = \frac{1}{\alpha} \log (1 + \alpha \xi_y \cdot \eta_x) - \varphi(\xi_y) - \psi(\eta_x).
\end{equation}
As $\alpha \downarrow 0$, these identities reduce to well-known properties of the Bregman divergence \cite[Chapter 1]{A16}. By analyzing carefully the primal and dual geodesics as well as the self-dual representation \eqref{eqn:L.alpha.self.dual}, we have the following result.

\begin{theorem} \label{thm:curvature}
For $t_1, t_2 > 0$ small, we have
\begin{equation} \label{eqn:H.expansion}
\begin{split}
H(t_1, t_2) &= -\langle \mathbf{v}, \mathbf{w} \rangle t_1 t_2  -\alpha \langle \mathbf{v}, \mathbf{w} \rangle \left[ \frac{\| \mathbf{v} \|^2}{3} t_1^3 t_2 + \frac{\|\mathbf{w} \|^2}{3} t_1 t_2^3 + \frac{\langle \mathbf{v}, \mathbf{w} \rangle}{2} t_1^2 t_2^2 \right] \\
 &\quad + \mathrm{higher}\text{ }\mathrm{order}\text{ }\mathrm{terms}.
\end{split}
\end{equation}
\end{theorem}
\begin{proof}
By \cite[Corollary 2]{W18}, the primal/dual geodesics of $L^{(\alpha)}$-divergence are straight lines in the primal/dual coordinate systems, up to time changes. Thus we can write $\xi_{r}(t_1) = \xi_q + s_1(t_1)v$ and $\eta_{p}(t_2) = \eta_q + s_2(t_2)w$, where $v$ and $w$ are the coordinate representations of ${\bf v}$ and ${\bf w}$, and $s_1$ and $s_2$ are time changes. For notational simplicity we suppress the parameters $t_1$ and $t_2$. Using \cite[(89)]{W18}, we have
\begin{equation} \label{eqn:riemannian.inner.product}
   \langle \mathbf{v},\mathbf{w} \rangle = \left(\frac{v \cdot w}{1 + \alpha (\xi_q \cdot \eta_q) } - \frac{\alpha}{(1 + \alpha (\xi_q \cdot \eta_q))^2}( \eta_q \cdot v) ( \xi_q \cdot w) \right).
\end{equation}
Differentiating \eqref{eqn:L.alpha.self.dual} and using \eqref{eqn:riemannian.inner.product}, we expand $H(t_1,t_2)$ in terms of $s_1$ and $s_2$:
\begin{equation} \label{eqn:H.s.expression}
\begin{split}
    H(t_1,t_2) =& - \langle \mathbf{v}, \mathbf{w} \rangle s_1s_2 + \frac{\alpha \langle \mathbf{v}, \mathbf{w}\rangle}{1 + \alpha (\xi_q \cdot \eta_q)} \cdot \left(( \eta_q \cdot v) s_1^2s_2 + ( \xi_q \cdot w) s_1s_2^2 \right) \\
    &+ (C_3 - \alpha C_1C_2) \alpha^2 (C_1^2 +C_2^2 + C_1C_2) -\frac{\alpha}{2} (C_3 - \alpha C_1C_2)^2 \\
    &+ \mathrm{higher}\text{ }\mathrm{order}\text{ }\mathrm{terms},
\end{split}
\end{equation}
where $C_1 = \frac{\eta_q \cdot v}{1 + \alpha (\xi_q \cdot \eta_q)} s_1$, $C_2 = \frac{ \xi_q \cdot w}{1 + \alpha (\xi_q \cdot \eta_q)} s_2$, and $C_3 = \frac{ v \cdot w}{1 + \alpha (\xi_q \cdot \eta_q)} s_1 s_2$.

On the other hand, the geodesic equations (see \cite[(86)]{W18}) give us, after some simplifications, Taylor expansions of $s_1$ and $s_2$:
\begin{equation} \label{eqn.s.t.expression.1}
    s_1(t_1) = t_1 + \alpha (D\varphi(q) \cdot v) t_1^2 + T_1  t_1^3 + O(t_1^4),
\end{equation}
\begin{equation} \label{eqn.s.t.expression.2}
    s_2(t_2) = t_2 + \alpha (D\psi(q) \cdot w) t_2^2 + T_2  t_2^3 + O(t_2^4),
\end{equation}
where $T_1 = \frac{1}{3}(4(\alpha (D\varphi(q) \cdot v))^2 + \alpha (v^{\top} D^2\varphi(q)v))$ and $T_2 = \frac{1}{3}(4(\alpha (D\psi(q) \cdot w))^2 + \alpha (w^{\top} D^2\psi(q)w))$. The proof is completed by combining \eqref{eqn:H.s.expression}, \eqref{eqn.s.t.expression.1} and \eqref{eqn.s.t.expression.2}. \qed
\end{proof}

This result gives a geometric interpretation of the negative sectional curvature $-\alpha$ in terms of the canonical local $L^{(\alpha)}$-divergence $\mathbf{D}$. Note that if we use another intrinsic divergence (such as the conformal divergence) we will get a different expression in \eqref{eqn:H.expansion}. Analogous results can be derived for the $L^{(-\alpha)}$-divergence.

Note that Theorem \ref{thm:curvature} implies that $\left. \frac{\partial^2}{\partial t_1^2 \partial t_2^2} {\bf D}[r(t_1) : p(t_2)] \right|_{t_1 = t_2 = 0} = -2 \alpha \langle {\bf v}, {\bf w} \rangle^2$, so the sectional curvature $-\alpha$ may be interpreted in terms of this fourth order mixed derivative. In \cite[Theorem 3.13]{WY19} we extended this result to any divergence. This is formulated using a novel connection between the information geometry of $c$-divergence (which covers all divergences) and the pseudo-Riemannian framework of Kim and McCann \cite{KM10}. In particular, for any divergence ${\bf D}$, the mixed derivative $\left. \frac{\partial^2}{\partial t_1^2 \partial t_2^2} {\bf D}[r(t_1) : p(t_2)] \right|_{t_1 = t_2 = 0}$ is equal to $-2$ times an un-normalized cross curvature of the Kim-McCann metric induced by the cost function. The reader is referred to \cite{WY19} for more details. To conclude this paper, let us remark that for a statistical manifold with non-constant sectional curvature, this cross sectional curvature is not intrinsic as there are different divergences (and hence Kim-McCann metrics) which induce the same dualistic structure. A natural starting point is to analyze the canonical divergence of Ay and Amari constructed in \cite{AA15}. We leave this as a problem for future research.

\bibliographystyle{plain}
\bibliography{geometry.ref}

\end{document}